\documentclass[review]{elsarticle}

\usepackage{lineno}
\usepackage{amsmath, amsthm, amscd, amsfonts, amssymb, graphicx}
 \usepackage{array,multirow}
 \usepackage{float}
 \usepackage{soul}
\usepackage[a4paper]{geometry}

\usepackage{mathrsfs}
\usepackage[dvipsnames,table,xcdraw]{xcolor}
\usepackage{xspace}
\usepackage{float}
\usepackage{subfig}
\newtheorem{thm}{Theorem}[section]

\newtheorem{defn}{Definition}[section]

\modulolinenumbers[5]

\usepackage{pifont}
\newcommand{\cmark}{\ding{51}}%
\newcommand{\xmark}{\ding{55}}%
\usepackage{booktabs,dcolumn}
\usepackage{tikz}
\usepackage[colorlinks=true]{hyperref}

\begin{document}
\makeatletter
\def\ps@pprintTitle{%
   \let\@oddhead\@empty
   \let\@evenhead\@empty
   \let\@oddfoot\@empty
   \let\@evenfoot\@oddfoot
}
\makeatother
\begin{frontmatter}

\title{Three-species predator-prey model with respect to Caputo and Caputo-Fabrizio fractional operators}

\author[add1]{Leila Eftekhari}
\ead{leila.eftekhari32@gmail.com}
\address[add1]{Department of Mathematics,  Tarbiat Modares University, Tehran, Iran.}

\author[addM]{Moein Khalighi\corref{cor1}}
\ead{moein.khalighi@utu.fi }
\address[addM]{Department of Future Technologies,  University of Turku, Turku, Finland.}

\author[add2]{Soleiman Hosseinpour}
\ead{soleiman.hosseinpour@gmail.com}
\address[add2]{Department of Applied Mathematics, Shahrood University of Technology, Shahrood, Iran.}

\author[addM]{Leo Lahti}
\ead{leo.lahti@utu.fi}

\cortext[cor1]{Corresponding author}
\begin{abstract}
We study distributed lag effects in three-dimensional Lotka-Volterra systems by applying the concept of fractional calculus. We derive a new numerical method that provides enhanced stability for the Caputo-Fabrizio operator based on Adams-Bashforth method, considering non-singular kernel in the definition of Caputo-Fabrizio operator. We investigate the stability conditions of this system with comparisons to the Caputo fractional derivative. Numerical results show that the type of differential operators and the value of orders significantly influence the stability of the numerical solution, and dynamics of the Lotka-Volterra system.

\end{abstract}
\begin{keyword}
Caputo-Fabrizio operator; Lotka-Volterra differential equations; Adam-Bashforth method; Stability analysis.
\end{keyword}

\end{frontmatter}

\section{Introduction}

The use of fractional calculus has rapidly increased in many fields of science and engineering~\cite{sun2018new,almeida2016modeling}.  
Fractional calculus brings more degrees of freedom for differentiation in the modeling of various phenomena, such as complex networks~\cite{PhysRevE.95.022409,10.1371/journal.pone.0154983}, optimal control problems \cite{soli1, soli2}, and Viscoelastic systems~\cite{matlob2019concepts}. Such flexible differential operators do not have unique definitions, however. The Grunwald-Letnikov, Riemann-Liouville, and Caputo definitions are examples of commonly used approaches that have been used in science and engineering. 

Here, we study the Caputo-Fabrizio (CF) fractional operator that at first comes from the definition of Caputo fractional derivative~\cite{fabrizio1}. It replaces the singular kernel of the Caputo derivative with an exponential function and has two representations for the temporal and spatial variables~\cite{fabrizio2}. This approach has been used, for instance, to model the behavior of the diffusion-convection equation, fractional Nagumo equation, and to control the wave on a shallow water \cite{apfabrizio, apfabrizio2}. The new operator has also been successfully applied in cancer treatment, HIV/AIDS infection, and tumor-obesity model
\cite{can, hiv, tom}. In \cite{ecofab}, a comprehensive overview of the CF operator has been conducted, showing that this operator is applied in economic and physical models.

Another significant application is provided by the classical Lotka-Volterra systems, which are sometimes called predator-prey or parasite-host equations. Such systems play a remarkable role in mathematical biology \cite{DAS20111}, and in financial systems, for example, biunivoc capital transfer from mother bank to subsiding bank and from subsiding bank to individuals or companies \cite{bank}. At first, these models were introduced independently by Alfred J. Lotka and Vito Volterra as a  simplified model of two species predator-prey population dynamics \cite{vol}. Whereas the classical formulation uses an integer-order differential, this is not always optimal due to the nonlocality of the interactions and the potential existence of memory or lag effects in the real systems. In 2007, Ahmed \textit{et al.} \cite{fvol} have introduced the fractional-order Lotka-Volterra system. Recent studies have generalized Lotka-Volterra models to  two-predator one-prey dynamics~\cite{twolot} and analysed a Lotka-Volterra fractional-order model using the Caputo fractional derivative \cite{cavol, AMIRIAN2020e04816}.

However, due to the appearance of a singularity in the definition of the Caputo fractional derivative, this operator is impractical for the modeling of some nonlocal dynamics~\cite{ghalib2020analytical}. Hence, the new nonsingular CF fractional operator has been proposed to overcome this shortcoming. Recently, scholars have examined the efficiency of this fractional framework through some real practical cases and provided more accurate parameter fitting than the classic integer and noninteger order models~\cite{baleanu2020new}. Importantly, Tarasov~\cite{ecofab} has shown that a system with the CF operator, in contrast to the Caputo fractional derivatives, cannot describe processes with memory effects but can suitably model processes with continuously distributed lag.

Motivated by the above discussion, we consider three-dimensional Lotka-Volterra differential equations described by the CF operator. We formulate a new corrected numerical method to solve the CF system.
We analyse the stability properties of the Lotka-Volterra model under Caputo and CF fractional operators. Consequently, we reveal new insights from differences in the stability region of these frameworks. To investigate the impact of different stability properties on the dynamic, we compare and illustrate some example cases.

\section{Definitions}\label{pre}

\subsection{Fractional calculus}
In this section, we outline the key definitions for fractional derivatives. The fractional derivative in the sense of Caputo is defined as follow \cite{kilbas2006theory}:
 
\begin{align}\label{eq1}
{}^{\text{C}}_{}D^{ \alpha}_{a^+} f(t)=\frac{1}{\Gamma(1-\alpha)}\int_{a}^{t}(t-\tau)^{-\alpha }f^\prime(\tau)d\tau,~~ 0<\alpha<1,~ t> a.
\end{align}
The kernel $ (t-\tau)^{-\alpha} $ in Eq. \eqref{eq1} cause a singularity at $ t=\tau $ that can be considered as a drawback in this definition. In 2015, Caputo and Fabrizio defined the following fractional derivative as \cite{fabrizio2}
\begin{align}\label{f1}
{}^{\text{CF}}_{}D^{ \alpha}_{a^+} f(t)=\frac{M(\alpha)}{1-\alpha}\int_a^t \exp(-\frac{\alpha}{1-\alpha}(t-\tau))f^\prime(\tau)d\tau, ~~t\geq a, 
\end{align} 
where $ f $ is a continuous and differentiable function on $ C^1[a, b]$ and $ M(\alpha) $ is a normalization  function such that $ M(0)=M(1)=1 $.

For more details on the above-mentioned fractional operators,  the readers are referred to~\cite{fabrizio1,fabrizio2}.
\subsection{Stability of the fractional-order system}
In this part of the paper, we recall some basic theorems to proceed with our goal on the stability of the Lotka-Volterra system with the CF operator.

\begin{defn}
The autonomous system, with $x(t_0)=x_0$
is asymptotically stable if and only if $\displaystyle\lim_{t\to +\infty}\| x(t)\|=0, $ where $ \Vert.\Vert $ is Euclidean norm.
\end{defn}

Consider the linear fractional-order autonomous system as follows
\begin{equation}\label{e11}
D^{\alpha}_{a^+}x(t)=Ax(t),
\end{equation}
where $x(t)\in \mathbb{R}^n$, $A\in\mathbb{R}^{n\times n}$, $0< \alpha< 1$ and $ D^{\alpha}_{a^+} $ is one of Caputo or CF operators.

\begin{thm}
The linear autonomous system \eqref{e11} with Caputo fractional  derivative for $ 0< \alpha< 1 $ is asymptotically stable if and only if 
$\left| \mathrm{arg} (\mathrm{spec}(A))\right| >\dfrac{\alpha \pi}{2}$, $\mathrm{spec}(A)$ is the spectrum (set of all eigenvalues) of $A$ \cite{Matignon96stabilityresults}. 
\end{thm}

\begin{thm}\label{thr1}
The system \eqref{e11} with CF operator is asymptotically stable if eigenvalues $\lambda(A)$ of the matrix A satisfy one of the following conditions  \cite{stability}
\begin{itemize}
\item[1)]
$\| (\lambda(A))\| \geqslant\dfrac{1}{1-\alpha}, \lambda(A)\neq \dfrac{1}{1-\alpha},$
\item[2)]
$\mathrm{Re}(\lambda(A))>\dfrac{1}{1-\alpha},$
\item[3)]
$\mathrm{Re}(\lambda(A))< 0,$
\item[4)]
$\left| \mathrm{Im}(\lambda(A))\right| >\dfrac{1}{2(1-\alpha)}.$
\end{itemize}
\end{thm}

\section{Modeling and its stability}\label{main}
One can find the stability conditions of this system with Caputo derivative in~\cite{reft}. In this section, we investigate the stability of the Lotka-Volterra system involving the CF operator and Table \ref{t1} summarizes all the required conditions for the stability in respect to the both operators.  
\subsection{Fractional Lotka-Volterra model}
In this study, we consider a three-species Lotka-Volterra model as follows 
\begin{align}\label{e1}
\begin{cases}
D^{\alpha}_{a^+}x(t)=x(t)(a_1-a_2x(t)-y(t)-z(t))\\
D^{\alpha}_{a^+}y(t)=y(t)((1-a_3)+a_4x(t))\\
D^{\alpha}_{a^+}z(t)=z(t)((1-a_5)+a_6x(t)+a_7y(t))	
\end{cases}
\end{align}
where $ 0<\alpha\leq1 $ and $ a_i>0, i=0, 1, \ldots, 7 $ and $ {}_0D^{\alpha}_t $ is one of the differential operators Caputo or CF with initial conditions 
\begin{align}
x(0)=x_0,~~ y(0)=y_0,~~ z(0)=z_0,
\end{align}
where $ x_0, y_0, z_0 \in \mathbb{R}^+. $ In this model $ x(t)\geq 0 $ represents the population of the prey, $ y(t)\geq 0$ and $ z(t)\geq 0 $ represent the population of predators at the time $ t $.

Now we consider system \eqref{e1} in a compact form as follows
\begin{align}\label{ca1}
\begin{cases}
D^{\alpha}_{a^+}\textbf{u}(t)=\textbf{F}(\textbf{u}(t))\qquad 0< t<\infty\\
\textbf{u}(0)=\textbf{u}_0,
\end{cases}
\end{align}
where $\textbf{u}(t)={(x(t),y(t),z(t))^{T}}\in\mathscr{L}[0,t']$, where $\mathscr{L}[0,t']$ be the set of all continuous vector $ \textbf{u}(t) $ defined on the interval $[0,t']$ ($t'	>0$) and  $\textbf{F}$ is a real-valued continuous vector function. Then system \eqref{e1} can be written in the form
\begin{align*}
D^{\alpha}_{a^+} \textbf{u}(t)=\textbf{A}\textbf{u}(t)+x(t)\textbf{B}\textbf{u}(t)+y(t)\textbf{C}\textbf{u}(t)+z(t)\textbf{D}\textbf{u}(t)
\end{align*}
where
\begin{align*}
\textbf{A}=\begin{bmatrix}
a_1& 0&0\\
0& 1-a_3& 0\\
0&0&1-a_5
\end{bmatrix},~~
\textbf{B}=\begin{bmatrix}
-a_2&0&0\\
0&a_4&0\\
0&0&a_6
\end{bmatrix},~~
\textbf{C}=\begin{bmatrix}
-1& 0&0\\
0& 0&0\\
0&0&a_7
\end{bmatrix},~~
\textbf{D}=\begin{bmatrix}
-1&0&0\\
0&0&0\\
0&0&0
\end{bmatrix}.
\end{align*}

\begin{thm}
For $\textbf{u}(t)\in\mathscr{L}[0,t']$, system \eqref{ca1} has a unique solution.
\end{thm}
\begin{proof} 

Let
 $ \textbf{F}(\textbf{u}(t))=\textbf{A}\textbf{u}(t)+u_1(t)\textbf{B}\textbf{u}(t)+u_2(t)\textbf{C}\textbf{u}(t)+u_3(t)\textbf{D}\textbf{u}(t) $ and  $ \textbf{F}(\textbf{v}(t))=\textbf{A}\textbf{v}(t)+v_1(t)\textbf{B}\textbf{v}(t)+v_2(t)\textbf{C}\textbf{v}(t)+v_3(t)\textbf{D}\textbf{v}(t) $ where $\textbf{F}(\textbf{u}(t))$, $\textbf{F}(\textbf{v}(t))\in\mathscr{L}[0,t']$. Since $ \textbf{u}(t)=(u_1,u_2,u_3) $, $ \textbf{v}(t)=(v_1,v_2,v_3)\in\mathscr{L}[0,t'] $ such that $ \textbf{u}(t)\neq \textbf{v}(t) $. The following inequality holds
{\small
\begin{align*}
&\left\| \textbf{F}(\textbf{u}(t))- \textbf{F}(\textbf{v}(t))\right\|\\
&=\bigg\| \textbf{A}\textbf{u}(t)+ u_1(t)\textbf{B}\textbf{u}(t)+u_2(t)\textbf{C}\textbf{u}(t)+u_3(t)\textbf{D}\textbf{u}(t)-(\textbf{A}\textbf{v}(t)+v_1(t)\textbf{B}\textbf{v}(t)+v_2(t)\textbf{C}\textbf{v}(t)+v_3(t)\textbf{D}\textbf{v}(t))\bigg\|\\
&\leqslant \left\| \textbf{A}(\textbf{u}(t)-\textbf{v}(t))\right\| +\left\| u_1(t)\textbf{B}(\textbf{u}(t)-\textbf{v}(t))\right\| +\left\| (u_1(t)-v_1(t))\textbf{B}\textbf{v}(t)\right\| +\left\| u_2(t) \textbf{C}(\textbf{u}(t)-\textbf{v}(t))\right\|\\
&+\left\| (u_2(t)-v_2(t))\textbf{C}\textbf{v}(t)\right\| +\left\| u_3(t)\textbf{D}(\textbf{u}(t)-\textbf{v}(t))\right\| + \left\| (u_3(t)-v_3(t))\textbf{D}\textbf{v}(t)\right\|\\
& \leqslant \bigg[\|\textbf{A}\|+\|\textbf{B}\| \left( |u_1(t)| +\| \textbf{v}(t)\| \right) +\|\textbf{C}\| \left( \left| u_2(t)\right| +\left\| \textbf{v}(t)\right\| \right) +\|\textbf{D}\| \left(\left| u_3(t)\right| +\left\| \textbf{v}(t)\right\| \right) \bigg]\times \left\| \textbf{u}(t) -\textbf{v}(t)\right\|,
\end{align*}}
then we have
\[\left\| \textbf{F}(\textbf{u}(t))-\textbf{F}(\textbf{v}(t))\right\| \leqslant \textbf{L}\left\| \textbf{u}(t)-\textbf{v}(t)\right\|,\]
where
\[\textbf{L}=\|\textbf{A}\|+(\|\textbf{B}\| +\|\textbf{C}\|+\|\textbf{D}\|)(M_1+M_2)>0,\]
and $ M_1 $  and $ M_2 $ are positive constant and satisfy $ \Vert \textbf{u}\Vert\leq M_1 $, $ \Vert \textbf{v}\Vert\leq M_2 $ as a aresult of $ \textbf{u}, \textbf{v} \in\mathscr{L}[0,t']. $ It means that $ \textbf{F} (X(t)) $ is continuous and satisfying Lipschitz condition, then
the initial value problem \eqref{ca1} has a unique solution.
\end{proof}

\subsection{Stability of the model}
In this subsection, we discuss the stability of non-linear Lotka-Volterra differential equations \eqref{e1} described by the CF operator. In the case of non-linear systems, we study the local stability of equilibrium points and the following theorems are presented to investigate the local stability of equilibrium points. In order to determine the equilibrium points of system \eqref{e1}, let us consider
\[D^{\alpha}_{a^+}x(t)=0,~~~ D^{\alpha}_{a^+}y(t)=0, ~~~D^{\alpha}_{a^+}z(t)=0.\]
The  equilibrium points of system \eqref{e1} are obtained and denoted as
\begin{align*}
&\varepsilon_0 =(0,0,0), \\
&\varepsilon_1=(\dfrac{a_1}{a_2}, 0,0), \\
&\varepsilon_2=\left( \dfrac{a_5-1}{a_6},0, \dfrac{a_1a_6-a_2(a_5-1)}{a_6}\right),\\
&\varepsilon_3 =\left(\dfrac{a_3-1}{a_4}, \dfrac{a_1a_4-a_2(a_3-1)}{a_4}, 0\right),\\
&\varepsilon_4 =\left(\dfrac{a_3-1}{a_4}, \dfrac{a_4(a_5-1)-a_6(a_3-1)}{a_7a_4}, \dfrac{a_4(1+a_1a_7-a_5)+(a_6-a_2a_7)(a_3-1)}{a_7a_4}\right).
\end{align*}
To adjust the conditions for the actual situations, the equilibrium points must be nonnegative. In this regard, it is obvious that
$ \varepsilon_0 $ and $ \varepsilon_1 $ always exist, and $\varepsilon_2$ exists when 
$a_3\geq 1$ and $a_1a_4\geq a_2(a_3-1)$, and it happens for $\varepsilon_3$ when 
$a_5\geq 1$ and $a_1a_6\geq a_2(a_5-1)$. Finally, the conditions $a_3\geq 1$, $a_4(a_5-1)\geq a_6(a_3-1)$ and $a_4\geq\dfrac{(a_2a_7-a_6)(a_3-1)}{(1+a_1a_7-a_5)}$ (or if $(1+a_1a_7-a_5)<0$ then $a_4\leq\dfrac{(a_6-a_2a_7)(a_3-1)}{(1+a_1a_7-a_5)}$, else if $(1+a_1a_7-a_5=0$ then $a_6>a_2a_7$) are necessary for the existence of $ \varepsilon_4 $.

\begin{thm}
Let $ \varepsilon^* $ be an equilibrium point of the nonlinear system \eqref{e1}, with CF operator, then  equilibrium point $ \varepsilon^* $ is asymptotically stable if eigenvalues of Jacobian matrix $\lambda(J(\varepsilon^*))$ satisfy one of the following conditions 
\begin{itemize}
\item[1)]
$\|\lambda (J(\varepsilon^*))\| \geqslant\dfrac{1}{1-\alpha}, \lambda(J(\varepsilon^*))\neq \dfrac{1}{1-\alpha},$
\item[2)]
$\mathrm{Re}(\lambda(J(\varepsilon^*)))>\dfrac{1}{1-\alpha},$
\item[3)]
$\mathrm{Re}(\lambda(J(\varepsilon^*)))< 0,$
\item[4)]
$\left| \mathrm{Im}(\lambda(J(\varepsilon^*)))\right| >\dfrac{1}{2(1-\alpha)}.$
\end{itemize}
\end{thm}
\begin{proof}
The proof is straightforward with Theorem \ref{thr1} and \cite{Matignon96stabilityresults}.
\end{proof}
To study the local stability of the equilibrium points such as $(x^{\ast}, y^{\ast}, z^{\ast})$ for system \eqref{e1} we provide the Jacobian matrix $J(x^{\ast}, y^{\ast}, z^{\ast})$ as follows
\[J(x^{\ast}, y^{\ast}, z^{\ast})=\begin{bmatrix}
a_1-2a_2x^{\ast}-y^{\ast}-z^{\ast}& -x^{\ast}& -x^{\ast}\\
a_4 y^{\ast}& 1-a_3+a_4x^{\ast}& 0\\
a_6z^{\ast}& a_7z^{\ast}& a_6x^{\ast}-a_5+a_7y^{\ast}+1
\end{bmatrix}.\]
\subsubsection{The first equilibrium} For $\varepsilon_0$, the Jacobian can be expressed as
\[J(\varepsilon_0)=\begin{bmatrix}
a_1& 0&0\\
0& 1-a_3& 0\\
0&0&1-a_5
\end{bmatrix},\]
where eigenvalues are $\lambda_1=a_1$, $\lambda_2=1-a_3$, $\lambda_3=1-a_5$. 
Since $a_1> 0$, $1-a_3< 0$, $1-a_4< 0$, then $\varepsilon_0$ is stable if $a_1>\dfrac{1}{1-\alpha}$.

\subsubsection{The second equilibrium} For $\varepsilon_1$, the Jacobian matrix is
\[J(\varepsilon_1)=\begin{bmatrix}
-a_1&-\dfrac{a_1}{a_2}&-\dfrac{a_1}{a_2}\\
0& 1-a_3+\dfrac{a_1a_4}{a_2}&0\\
0&0& 1-a_5+\dfrac{a_1a_6}{a_2}
\end{bmatrix},\]
where $\lambda_1=-a_1< 0$, $\lambda_2=1-a_3+\dfrac{a_1a_4}{a_2}$, and $\lambda_3=1-a_5+\dfrac{a_1a_6}{a_2}$. Thus, $\varepsilon_1$ is asymptotically stable when 
\begin{align*}
a_1a_4<a_2a_3-a_2,\\
a_1a_6<a_2a_5-a_2,
\end{align*}
or
\begin{align*}
a_2(1-a_3)(1-\alpha)>a_2-a_1a_4(1-\alpha),\\
a_2(1-a_5)(1-\alpha)>a_2-a_1a_6(1-\alpha).
\end{align*}
\subsubsection{The third equilibrium} For $\varepsilon_2$ the Jacobian matrix is 
\[J(\varepsilon_2)=\begin{bmatrix}
\dfrac{a_2}{a_6}(1-a_5) & \dfrac{1-a_5}{a_6}& \dfrac{1-a_5}{a_6}\\
0& 1-a_3-\dfrac{a_4}{a_6}(1-a_5)& 0\\
a_1a_6+a_2(1-a_5)& \dfrac{a_7}{a_6}\left( a_1a_6+a_2(1-a_5)\right) &0
\end{bmatrix},\]
we use the below notation
\[J(\varepsilon_2)=\begin{bmatrix}
A& B&C\\
0& D&0\\
E& F&0
\end{bmatrix}.\]
The characteristic equation is as follows
\[(\lambda-D)(\lambda^2-A\lambda -CE)=0,\]
by the condition $\dfrac{a_5-1}{a_6}<\dfrac{a_1}{a_2}<\dfrac{a_3-1}{a_4}$ we get 
\[A< 0, C< 0, E> 0, D< 0,\]
therefore
\[\lambda_1=D< 0, \lambda_2+\lambda_3=A< 0, \lambda_2\lambda_3=-CE> 0.\]
The eigenvalues are \[ \lambda_1=\left[ 1-a_3-\dfrac{a_4}{a_6}(1-a_5)\right], \] and \[ \lambda_{2,3}=\frac{\bigg[ a_2(1-a_5)\pm \sqrt{a_2^2 (1-a_5)^2+4a_6(1-a_5)(a_1a_6+a_2(1-a_5))}\bigg]}{2a_6}. \]
In this case, we can conclude that $\varepsilon_2$ is locally asymptoialy stable. However, when the condition $\dfrac{a_5-1}{a_6}<\dfrac{a_1}{a_2}<\dfrac{a_3-1}{a_4}$ was not available, $\varepsilon_2$ could be locally asymtotically stable when
$\lambda_1, \lambda_2, \lambda_3>\dfrac{1}{1-\alpha}$, which leads 

\[\left[ 1-a_3-\dfrac{a_4}{a_6}(1-a_5)\right] (1-\alpha)> 1,\]
and
\[\bigg[ a_2(1-a_4)\pm \sqrt{a_2^2 (1-a_4)^2+4a_6(1-a_5)(a_1a_6+a_2(1-a_5))}\bigg] (1-\alpha)> 2a_6.\]

\subsubsection{The fourth equilibrium} Jacobian of $\varepsilon_ 3$ is
\[J(\varepsilon_3)=\begin{bmatrix}
\dfrac{a_2}{a_4}(1-a_3)& \dfrac{1-a_3}{a_4}& \dfrac{1-a_3}{a_4}\\
a_1a_4+a_2(1-a_3)& 0&0\\
0&0& w
\end{bmatrix},\]
where $w=1-a_5-\dfrac{a_6}{a_4}(1-a_3)+\dfrac{a_7}{a_4}[a_1a_4+a_2(1-a_3)]$. Same as $\varepsilon_2$, we can provid the stability condition as
\[\dfrac{a_3-1}{a_4}<\dfrac{a_1}{a_2}<\dfrac{a_5-1}{a_6},\]
where $\lambda_1=1-a_4-\dfrac{a_6}{a_4}(1-a_3)+\dfrac{a_7}{a_4}[a_1a_4+a_2(1-a_3)]$ and 
\[\lambda_{2,3}=\dfrac{
\bigg[ a_2(1-a_3)\pm \sqrt{a_2^2(1-a_3)^2+4a_4(1-a_3)[a_1a_4+a_2(1-a_3)]}\bigg]}{2a_4}.\]
When the condition $\dfrac{a_3-1}{a_4}<\dfrac{a_1}{a_2}<\dfrac{a_5-1}{a_6}$
is not available, $\varepsilon_3$ is locally asymptotically stable when $\lambda_1(1-\alpha)> 1$, $\lambda_2(1-\alpha)> 1$, $\lambda_3(1-\alpha)> 1$.

\subsubsection{The fifth equilibrium} For $\varepsilon_4$, Jacobian matrix is as follows 
\begin{align*}
J(\varepsilon_4)&=\begin{bmatrix}
A& B&B\\
C&0&0\\
D&E&0
\end{bmatrix},
\end{align*}
where
\begin{align*}
&A=\dfrac{a_2}{a_4}(1-a_3), \\
&B=\dfrac{1-a_3}{a_4}, \\
&C=-\dfrac{a_4-a_6+a_3a_6-a_4a_5}{a_7},\\
&D=\dfrac{a_6(a_4-a_6+a_2a_7+a_3a_6-a_4a_5+a_1a_4a_7-a_2a_3a_7)}{a_4a_7},\\
&E=\dfrac{a_4-a_6+a_2a_7+a_3a_6-a_4a_5+a_1a_4a_7-a_2a_3a_7}{a_4}.
\end{align*}
To compute eigenvalues of the above matrix, we consider the characteristic polynomial,  $L(\lambda)=\lambda^3+a\lambda^2+b\lambda +c, $ where $ a=-A, b=-B(C+D), c=-BCE. $ It is obvious $ a, c>0. $ If $ a_1a_2>a_3 $ then Routh-Hurwitz criterion shows that the all roots of        are negative. The equation $ a_1a_2-a_3=B[A(C+D)+CE] $ is positive if 
\[a_6>\dfrac{a_2a_4(a_3-1)[w+a_2(a_3-1)]}{w(a_2+a_4)+a_2a_4(a_3-1)},
\]
where
\[
w=a_4(1+a_1a_7-a_5)+(a_6-a_2a_7)(a_3-1).
\]
To investigate the stability of the system in the sense of CF operator we consider the following parameter for $ L(\lambda) $ as
\begin{align*}
p&=b-\dfrac{a}{3},\\
q&=\dfrac{2a^3}{27}-\dfrac{ab}{3}+c,\\
\Delta &=\dfrac{q^2}{4}+\dfrac{p^3}{27}.
\end{align*}
If $\Delta> 0$, then we have only one real solution
\begin{align}\label{er7}
\lambda=\left(-\dfrac{q}{2}+\sqrt{\Delta}\right)^{\frac{1}{3}}+\left( -\dfrac{q}{2}-\sqrt{\Delta}\right)^{\frac{1}{3}}-\dfrac{q}{3}.
\end{align}
If $\Delta=0$, there are repeated roots
\begin{align}
\lambda_1=-2(\dfrac{q}{2})^{\frac{1}{3}}-\dfrac{q}{3}\qquad \lambda_2=\lambda_3=(\dfrac{q}{2})^{\frac{1}{3}}-\dfrac{q}{3}.
\end{align}
If $\Delta> 0$ then roots are same as below
\begin{align}
&\lambda_1=\dfrac{2\sqrt{p}}{\sqrt{3}}sin(\dfrac{1}{3}arc sin(\dfrac{3\sqrt{3}q}{2(\sqrt{-p})^3}))-\dfrac{a}{3},\\
&\lambda_2=-\dfrac{2\sqrt{p}}{\sqrt{3}}sin(\dfrac{1}{3}arc sin(\dfrac{3\sqrt{3}q}{2(\sqrt{-p})^3}+\dfrac{\pi}{3}))-\dfrac{a}{3},\\ \label{er11}
&\lambda_3=\dfrac{2\sqrt{-p}}{\sqrt{3}}cos(\dfrac{1}{3}arc sin(\dfrac{3\sqrt{3}q}{2(\sqrt{-p})^3}+\dfrac{\pi}{6}))-\dfrac{a}{3}.
\end{align}
Consequently, $ \varepsilon_4 $ is locally asymptotically stable if in any case all of the eigenvalues satisfying these conditions
\[\lambda_1, \lambda_2, \lambda_3>\dfrac{1}{(1-\alpha)}.\]
We end this section by  summarizing the stability conditions of all the equilibrium points for Caputo and CF operators in Table \ref{t1}.

\begin{table}[t] 
{\scriptsize{
\begin{center}
\caption{Stability conditions for Caputo and CF operators.}\label{t1}
\begin{tabular}{ |m{1.45cm}| m{4.4cm} |m{7.3cm}|}
\hline \rowcolor{gray!15}
\textbf{Equilibrium point} & \textbf{Caputo derivative}  & \textbf{CF operator}  \\ \hline
$ \varepsilon_0 $&Always saddle &$ a_1>\dfrac{1}{1-\alpha} $\\\hline

$ \varepsilon_1 $&$  a_1a_2<a_2a_3-a_2$ \newline~~ and ~~\newline $  a_1a_2<a_2a_3-a_2$&$  a_1a_2<a_2a_3-a_2$  and $  a_1a_2<a_2a_3-a_2$ \newline or \newline $ \dfrac{a_1a_4-a_2a_3}{a_2}>\dfrac{\alpha}{1-\alpha} $ ~~and~~ $ \dfrac{a_1a6-a_2a_5}{a_2}>\dfrac{\alpha}{1-\alpha} $\\ \hline

$ \varepsilon_2 $&$ \dfrac{a_5-1}{a_6}<\dfrac{a_1}{a_2}<\dfrac{a_3-1}{a_4} $&$ \dfrac{a_5-1}{a_6}<\dfrac{a_1}{a_2}<\dfrac{a_3-1}{a_4} $\newline or \newline $ \left[ 1-a_3-\dfrac{a_4}{a_6}(1-a_5)\right]>\dfrac{1}{1-\alpha} $~~and \newline $ \frac{\bigg[ a_2(1-a_5)\pm \sqrt{a_2^2 (1-a_5)^2+4a_6(1-a_5)(a_1a_6+a_2(1-a_5))}\bigg]}{2a_6}>\frac{1}{1-\alpha} $\\\hline

$ \varepsilon_3 $&$ \dfrac{a_3-1}{a_4}<\dfrac{a_1}{a_2}<\dfrac{a_5-1}{a_6} $&$ \dfrac{a_3-1}{a_4}<\dfrac{a_1}{a_2}<\dfrac{a_5-1}{a_6} $\newline or \newline $ 1-a_4-\dfrac{a_6}{a_4}(1-a_3)+\dfrac{a_7}{a_4}[a_1a_4+a_2(1-a_3)]>\dfrac{1}{1-\alpha} $ ~~  and \newline $ \frac{
\bigg[ a_2(1-a_3)\pm \sqrt{a_2^2(1-a_3)^2+4a_4(1-a_3)[a_1a_4+a_2(1-a_3)]}\bigg]}{2a_4}>\frac{1}{1-\alpha} $ \\\hline

$ \varepsilon_4 $&\scriptsize{$ a_6>\dfrac{a_2a_4(a_3-1)[w+a_2(a_3-1)]}{w(a_2+a_4)+a_2a_4(a_3-1)} $}\newline $ \big(w=a_4(1+a_1a_7-a_5)+(a_6-a_1a_7)(a_3-1)\big) $&$ a_6>\dfrac{a_2a_4(a_3-1)[w+a_2(a_3-1)]}{w(a_2+a_4)+a_2a_4(a_3-1)} $ \newline or \newline $ \lambda_1, \lambda_2, \lambda_3>\dfrac{1}{(1-\alpha)} $\newline (see equations \eqref{er7}-\eqref{er11})\\\hline
\end{tabular}
\end{center}
}}
\end{table}

\section{Numerical algorithm}\label{algo}
 
The Predictor-Corrector methods are well-known numerical approach so that their extensions can provide accurate numerical solutions of fractional differential equations (FDEs). For instance, in \cite{dadras}, a numerical method based on Adams-Bashforth methods is proposed for solving FDEs with Caputo derivative. In \cite{adam2}, authors have investigated a fractional Adams-Bashforth method for solving FDEs with the CF operator, although their arguments are flawed.  For this aim, in this part of the paper, we correct this method to solve the Lotka-Volterra system of Caputo and the CF operator and we compare the solutions using the both operators. Consider the following differential equation:

\begin{align}
&{}^{CF}D_0^{\alpha}f(x)=g(t, f(x)),~~x\in[0, t'],\\
&f^{(i)}(0)=f_0^{i}, ~~i=0, 1, 2, \ldots, n-1, ~n=\lceil \alpha\rceil,
\end{align}
which is equivalent to the following equation
\begin{equation}
f(x)=T_{n-1}(x)+\dfrac{1-\alpha}{M(\alpha)(n-2)!}\int_0^x(x-t)^{n-2}g(t, f(t))dt+\dfrac{\alpha}{M(\alpha)(n-1)!}\int_0^x(x-t)^{n-1}g(t, f(t))dt
\end{equation}
where $ T_{n-1}(x) $ is the Taylor expansion of $ f(x) $ centered at $ x_0=0 $ and $ T_{n-1}(x)=\sum_{i=0}^{n-1}\frac{x^i}{i!}f_0^{(i)} $.

The corrector formula $f_{k+1}$ can be written as follows
\begin{align*}
f_{k+1}&= T_{n-1}(x)+\dfrac{\alpha}{M(\alpha)(n-1)!}\bigg[ \sum_{i=0}^ k b_{i,k+1}g(x_i, f_i)+ b_{k+1, k+1}g(x_{k+1}, f_{k+1}^p)\bigg]\\
b_{i,k+1}&=\dfrac{h^n}{n(n+1)}\begin{cases}
k^{n+1}-(k+1)^n (k-n),& i=0\\
(k-i-2)^{n+1}-2(k-i+1)^{n+1}+(k-i)^{n+1},& 1\leqslant i\leqslant k\\
1,& i=k+1
\end{cases}
\end{align*}
and by the fractional Adams-Bashforth-multon method \cite{dadras}, $f_{k+1}^p$ is determined by
\[f_{k+1}^p =T_{n-1}(x)+\dfrac{\alpha}{M(\alpha)(n-1)!}\sum_{i=0}^k d_{i, k+1}g(x_i, f_i)\]
where
\[d_{i,k+1}=\dfrac{h^n}{n}\bigg[ (k-i+1)^n -(k-i)^n\bigg]\]

Now, consider the following fractional-order system involving CF operator
\begin{equation}\label{eq4}
\begin{cases}
{}^{CF}D^{\alpha}_{a^+}x(t)=f_1(x,y,z),&\\
{}^{CF}D^{\alpha}_{a^+}y(t)=f_2(x,y,z),&\\
{}^{CF}D^{\alpha}_{a^+}z(t)=f_3(x,y,z).&
\end{cases}
\end{equation}
We consider $0\leqslant\alpha\leqslant 1$ for simplicity and assume that $(x_0, y_0, z_0)$ is the initial point. Applying the above scheme, system \eqref{eq4} can be discretized as follows
\begin{align*}
x_{k+1}&=x_0+\dfrac{\alpha}{M(\alpha)(n-1)!}\bigg[ \sum_{i=0}^k b_{1\;i,k+1} f_1(x_i,y_i, z_i)+b_{1\;k+1, k+1}f_1(x_{k+1}^p , y_{k+1}^p , z_{k+1}^p)\bigg],\\
y_{k+1}&=y_0+\dfrac{\alpha}{M(\alpha)(n-1)!}\bigg[ \sum_{i=0}^k b_{2\; i,k+1}f_2(x_i, y_i, z_i)+b_{2\;k+1, k+1}f_2(x_{k+1}^p, y_{k+1}^p , z_{k+1}^p)\bigg],\\
z_{k+1}&=z_0+\dfrac{\alpha}{M(\alpha)(n-1)!}\bigg[ \sum_{i=0}^k b_{3\; i,k+1}f_3(x_i, y_i, z_i)+b_{3\; k+1, k+1}f_3(x_{k+1}^p, y_{k+1}^p , z_{k+1}^p)\bigg],
\end{align*}
where
\begin{align}
x_{k+1}^p&=x_0+\dfrac{\alpha}{M(\alpha)(n-1)!}\bigg[ \sum_{i=0}^k d_{1\; i,k+1}f_1(x_i, y_i, z_i)\bigg],\\
y_{k+1}^p&=y_0+\dfrac{\alpha}{M(\alpha)(n-1)!}\bigg[ \sum_{i=0}^k d_{2\; i,k+1}f_2(x_i, y_i, z_i)\bigg],\\
z_{k+1}^p&=z_0+\dfrac{\alpha}{M(\alpha)(n-1)!}\bigg[ \sum_{i=0}^k d_{3\; i,k+1}f_3(x_i, y_i, z_i)\bigg],
\end{align}
and
\begin{align*}
b_{j\; i,k+1}&=\dfrac{h^n}{n(n+1)}\begin{cases}
k^{n+1}-(k+1)^n (k-n),& i=0\\
(k-i+2)^{n+1}-2(k-i+1)^{n+1}+(k-i)^{n+1},& 1\leqslant i\leqslant 1\\
1,& i=k+1
\end{cases}\\
d_{j\; i,k+1}&=\dfrac{h^n}{n}\left[ (k-i+1)^n-(k-i)^n\right].
\end{align*}
In the following, we apply the proposed numerical technique for simulation the solutions of system \ref{e1}. It is worth mentioning that to obtain the numerical results in the sense of Caputo derivative we use the equivalent Adams-Bashforth-multon method described in \cite{dadras}.
\section{Numerical implementation}\label{imp}
In this part, we discuss the numerical results of the fractional-order Lotka-Volterra model \eqref{e1} with Caputo and CF operators, by using the numerical method described in Sec. \ref{algo}. It is helpful to classify the numerical results according to the positions of the eigenvalues and discuss the behavior of the system in the sense of Caputo and CF operator. To illustrate such a classification, we provide stability and instability region for both operators in Fig. \ref{fig:compare} and set four eigenvalues on the plane for different cases.  
The unstable domain of the system with Caputo derivative is an unbound region limited by two lines with angles $-\alpha\frac{\pi}{2}$ and $\alpha\frac{\pi}{2}$. On the other hand, the unstable domain of the system with CF operator is a bounded closed circle centered at $(0,\frac{1}{2(1-\alpha)})$ with the radius $\frac{1}{2(1-\alpha)}$. For both cases, it is clear that the stability of the system has an inverse relation with the order derivative; in fact, the smaller $\alpha$ is, the more there is space for stablitity and vice versa. As it is shown in Fig \ref{fig:compare}, the eigenvalues can be located in four distinct classes: $\lambda_A$ is in an area where both systems are stable; the class of $\lambda_B$ is where the system with Caputo derivatives is stable, but the system with CF operator is not stable; $\lambda_C$ denotes a class of eigenvalues staying at where both systems are unstable; and finally, $\lambda_D$ is where the system in the sense of CF operator is stable but the system with Caputo derivatives is not stable. 
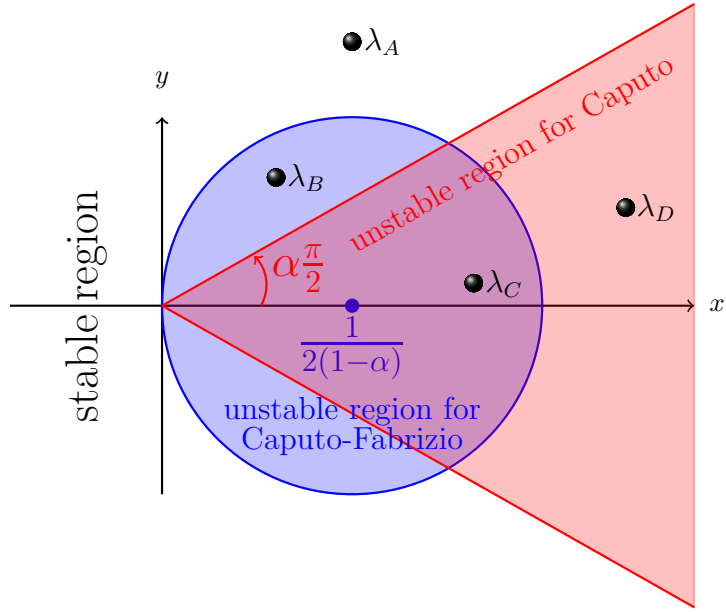
\begin{figure}
\centering
\begin{tikzpicture}[thick]
\draw [->](-2,0) -- (7,0);
\draw [->](0,-2.5) -- (0,2.5);
\node at (0,3) {$y$};
\node at (7.3,0) {$x$};

\filldraw [blue,nearly transparent](2.5,0) circle (2.5cm);
\draw [blue,thick](2.5,0) circle (2.5cm);

\filldraw[blue] (2.5,0) circle (.08cm);
\node[text=blue,font=\LARGE] at (2.5,-0.55) {$\frac{1}{2(1-\alpha)}$};

\filldraw[red,nearly transparent] (0,0) -- (7,-4) -- (7,4) -- (0,0);
\draw[red,thick] (0,0) -- (7,-4);
\draw[red,thick] (0,0) -- (7,4);
\draw [red,->] (1.3,0) arc (-30:45:15.5pt);
\node [text=red, font=\LARGE] at (1.8,.5) {$\alpha\frac{\pi}{2}$};

\node [font=\large] at (2.9,3.5) {$\lambda_{A}$};
\shade[ball color=black] (2.5,3.5) circle (.8ex);

\node [font=\large] at (1.9,1.7) {$\lambda_{B}$};
\shade[ball color=black] (1.5,1.7) circle (.8ex);

\node [font=\large] at (4.5,0.3) {$\lambda_{C}$};
\shade[ball color=black] (4.1,0.3) circle (.8ex);

\node [font=\large] at (6.5,1.3) {$\lambda_{D}$};
\shade[ball color=black] (6.1,1.3) circle (.8ex);

\node[text=blue,font=\large]  at (2.5,-1.4) {unstable region for};
\node[text=blue,font=\large]  at (2.5,-1.8) {Caputo-Fabrizio};

\node[text=red,rotate=30,font=\large]  at (4.6,2) {unstable region for Caputo};

\node[rotate=90,font=\LARGE]  at (-1,0) {stable region};
\end{tikzpicture}
\caption{Comparison stability and unstability domain of Caputo and CF operators}
\label{fig:compare}
\end{figure}

We collect a summary of the three examples in Table \ref{tab:table} to easily compare the behavior of the model concerning the Caputo and CF operators. 

\begin{table}
     \caption{The summary of examples; \textbf{C}, \textbf{CF}, and $U(0)$ denote Caputo, Caputo-Fabrizio, and initial values, respectively, and the notation {\large\cmark} indicates the system is asymptotically stable, while {\large\xmark} implys unstability.}
     \label{tab:table}
     \centering
     \small
\resizebox{14cm}{!}{
\begin{tabular}{|m{1.5cm}|m{1.2cm}|m{2.26cm}|m{4cm}|c|c|c|c|}
\hline
 \multicolumn{4}{|c|}{\cellcolor{gray!30}\large \textbf{Example 1}}&{\cellcolor{gray!15}\textbf{C}}&{\cellcolor{gray!15}\textbf{CF}}&{\cellcolor{gray!15}\textbf{C}}&{\cellcolor{gray!15}\textbf{CF}}\\ \cline{1-8}

{\cellcolor{gray!15}Coefficient}& {\cellcolor{gray!15}$U(0)$}& {\cellcolor{gray!15}Equilibrium}& {\cellcolor{gray!15}Eigenvalues} &\multicolumn{2}{c|}{$ \alpha=0.98 $}&\multicolumn{2}{c|}{$ \alpha\leq0.66 $}\\ \cline{1-8}

\multirow{5}{*}{\vtop{\hbox{\strut $a_1=3$}\hbox{\strut $a_2=0.5$}\hbox{\strut $a_3=4$}\hbox{\strut $a_4=3$}\hbox{\strut $a_5=4$}\hbox{\strut $a_6=9$}\hbox{\strut $a_7=4$}}}&\multirow{5}{*}{\vtop{\hbox{\strut $x_0=0.5$}\hbox{\strut $y_0=0.9$}\hbox{\strut $z_0=0.1$}}}&$ \varepsilon_0 (0,0,0)$&$ \lambda_0 (-3,-3,3) $& \large\xmark & \large\xmark & \large\xmark & \large \cmark \\ \cline{3-8}

& & $ \varepsilon_1 (6,0,0) $&$ \lambda_1 (-3,15,51) $&\large\xmark &\large\xmark &\large\xmark & \large \cmark \\ \cline{3-8}

& & $ \varepsilon_2 (0.33, 0, 2.83) $&$ \lambda_2 (-0.083-2.914i,  -0.083+2.914i, -2+0i) $& \large \cmark & \large \cmark & \large \cmark & \large \cmark \\ \cline{3-8}

& & $ \varepsilon_3 (1,2.5,0) $&$ \lambda_3 (-0.25-2.727i,
-0.25+2.727i, 16+0i) $&\large\xmark & \large\xmark &\large\xmark & \large \cmark \\ \cline{3-8}

& & {$ \varepsilon_4 (1,-1.5,4)$ Not Acceptable} &$ \lambda_4 (-1.239-5.904i, -1.239+5.904i, 1.978+0i) $&\large\xmark &\large\xmark &\large\xmark &\large\xmark \\ \cline{1-8}
\multicolumn{4}{|c|}{\cellcolor{gray!30}\large \textbf{Example 2}}&\multicolumn{2}{c|}{\cellcolor{gray!15}{\textbf{C}}}&\multicolumn{2}{c|}{\cellcolor{gray!15}{\textbf{CF}}}\\ \cline{1-8}
{\cellcolor{gray!15}{Coefficient}}& {\cellcolor{gray!15}{\cellcolor{gray!15}$U(0)$}}& {\cellcolor{gray!15}{Equilibrium}}& {\cellcolor{gray!15}{Eigenvalues}} &\multicolumn{4}{c|}{$ \alpha=0.6 $}\\ \cline{1-8}

\multirow{5}{*}{\vtop{\hbox{\strut $a_1=3$}\hbox{\strut $a_2=0.5$}\hbox{\strut $a_3=4$}\hbox{\strut $a_4=3$}\hbox{\strut $a_5=14$}\hbox{\strut $a_6=9$}\hbox{\strut $a_7=4$}}}&\multirow{5}{*}{\vtop{\hbox{\strut $x_0=2$}\hbox{\strut $y_0=2$}\hbox{\strut $z_0=3$}}}&$ \varepsilon_0 (0,0,0) $&$ \lambda_0 (-13,-3,3) $&\multicolumn{2}{c|}{\large\xmark}&\multicolumn{2}{c|}{\large \cmark }\\ \cline{3-8}

& & $ \varepsilon_1 (6,0,0) $&$ \lambda_1 (-3,15,41) $&\multicolumn{2}{c|}{\large\xmark}&\multicolumn{2}{c|}{\large \cmark}\\ \cline{3-8}

& & $ \varepsilon_2 (1.44,0,2.28) $&$ \lambda_2 (-0.361-5.429i,
-0.361+5.429i, 1.333+0i) $&\multicolumn{2}{c|}{\large\xmark}&\multicolumn{2}{c|}{\large\xmark}\\ \cline{3-8}

& & $ \varepsilon_3 (1,2.5,0) $&$ \lambda_3 (-0.25-2.727i, -0.25+2.727i, 6+0i) $&\multicolumn{2}{c|}{\large\xmark}&\multicolumn{2}{c|}{\large \cmark}\\ \cline{3-8}

& & $ \varepsilon_4 (1,1,1.5) $&$ \lambda_4 (0.276-4.123i,
0.276+4.123i, -1.053+0i) $&\multicolumn{2}{c|}{\large \cmark}&\multicolumn{2}{c|}{\large\xmark}\\ \cline{1-8}


\multicolumn{4}{|c|}{\cellcolor{gray!30}\large \textbf{Example 3}}&\multicolumn{2}{c|}{\cellcolor{gray!15}{\textbf{C}}}&\multicolumn{2}{c|}{\cellcolor{gray!15}{\textbf{CF}}}\\ \cline{1-8}
{\cellcolor{gray!15}{Coefficient}}& {\cellcolor{gray!15}{$U(0)$}}& {\cellcolor{gray!15}{Equilibrium}}& {\cellcolor{gray!15}{Eigenvalues}} &\multicolumn{4}{c|}{$ \alpha=0.4 $}\\ \cline{1-8}

\multirow{5}{*}{\vtop{\hbox{\strut $a_1=8$}\hbox{\strut $a_2=0.5$}\hbox{\strut $a_2=4$}\hbox{\strut $a_4=1$}\hbox{\strut $a_5=7$}\hbox{\strut $a_6=9$}\hbox{\strut $a_7=4$}}}&\multirow{5}{*}{\vtop{\hbox{\strut}\hbox{\strut $x_0=0.5$}\hbox{\strut $y_0=0.1$}\hbox{\strut $z_0=5$}}}&$ \varepsilon_0 (0,0,0) $&$ \lambda_0 (-6,-3,8) $&\multicolumn{2}{c|}{\large\xmark}&\multicolumn{2}{c|}{\large \cmark}\\
 \cline{3-8}

& & $ \varepsilon_1 (160,0,0) $&$ \lambda_1 (-8,157,1434) $&\multicolumn{2}{c|}{\large\xmark}&\multicolumn{2}{c|}{\large \cmark}\\ \cline{3-8}

& & $ \varepsilon_2 (0.666, 0, 7.966) $ & $ \lambda_2 (-0.016 -6.913i, -0.016 +6.913i, -2.333+0i) $&\multicolumn{2}{c|}{\large \cmark}&\multicolumn{2}{c|}{\large \cmark}\\ \cline{3-8}

& & $ \varepsilon_3 (3, 7.85, 0) $&$ \lambda_3 (-0.075 -4.852i, -0.075 +4.852i, 52.4 +0i) $&\multicolumn{2}{c|}{\large\xmark}&\multicolumn{2}{c|}{\large \cmark}\\ \cline{3-8}

& & {$ \varepsilon_4 (3, -5.25, 13.1) $ Not Acceptable} &$ \lambda_4 (-1.274 -18.50i, -1.274 +18.50i, 2.398 +0i) $&\multicolumn{2}{c|}{\large\xmark}&\multicolumn{2}{c|}{\large \cmark}\\ \cline{1-8}
\end{tabular}
}
\end{table}

\subsection{Example 1}
As one can see in Table \ref{tab:table}, the parameters of this example give five distinct equilibrium points (and corresponding eigenvalues), while equilibrium $\epsilon_4$ is not acceptable since it has a negative value. Thus, we should expect negative-value solutions of the system when we do not impose any constraints on the components. There is a recommended paper \cite{attraction} to avoid going toward such meaningless solutions and getting a feasible solution. 

Furthermore, this example shows the stability of the equilibrium points depends on the value of the fractional-order $\alpha$. As we expect from Fig. \ref{fig:compare}, the number of stable equilibrium points increases when we reduce the value of $\alpha$. In this case, when $\alpha$ is $0.98$ for both operators, the system is asymptotically stable only at $\epsilon_2$ (see Table \ref{fig:example1}). Indeed, Fig. \ref{fig:example1} (left) shows that the system gets steady at $\epsilon_2$, with different oscillations which are related to the definition of the operators. Nonetheless, the condition $\alpha \leq 0.66$ provides a larger area for the stability of the system so that three eigenvalues $\lambda_0$, $\lambda_1$, and $\lambda_3$ stay in the class of $\lambda_D$ (see Fig. \ref{fig:compare} and Table \ref{tab:table}). Hence, with appropriate initial values and differential orders, the system could converge to $\epsilon_3$ (see Fig. \ref{fig:example1}, right).

\begin{figure}[!tbp]%
\centering
\subfloat{{\includegraphics[scale=.45]{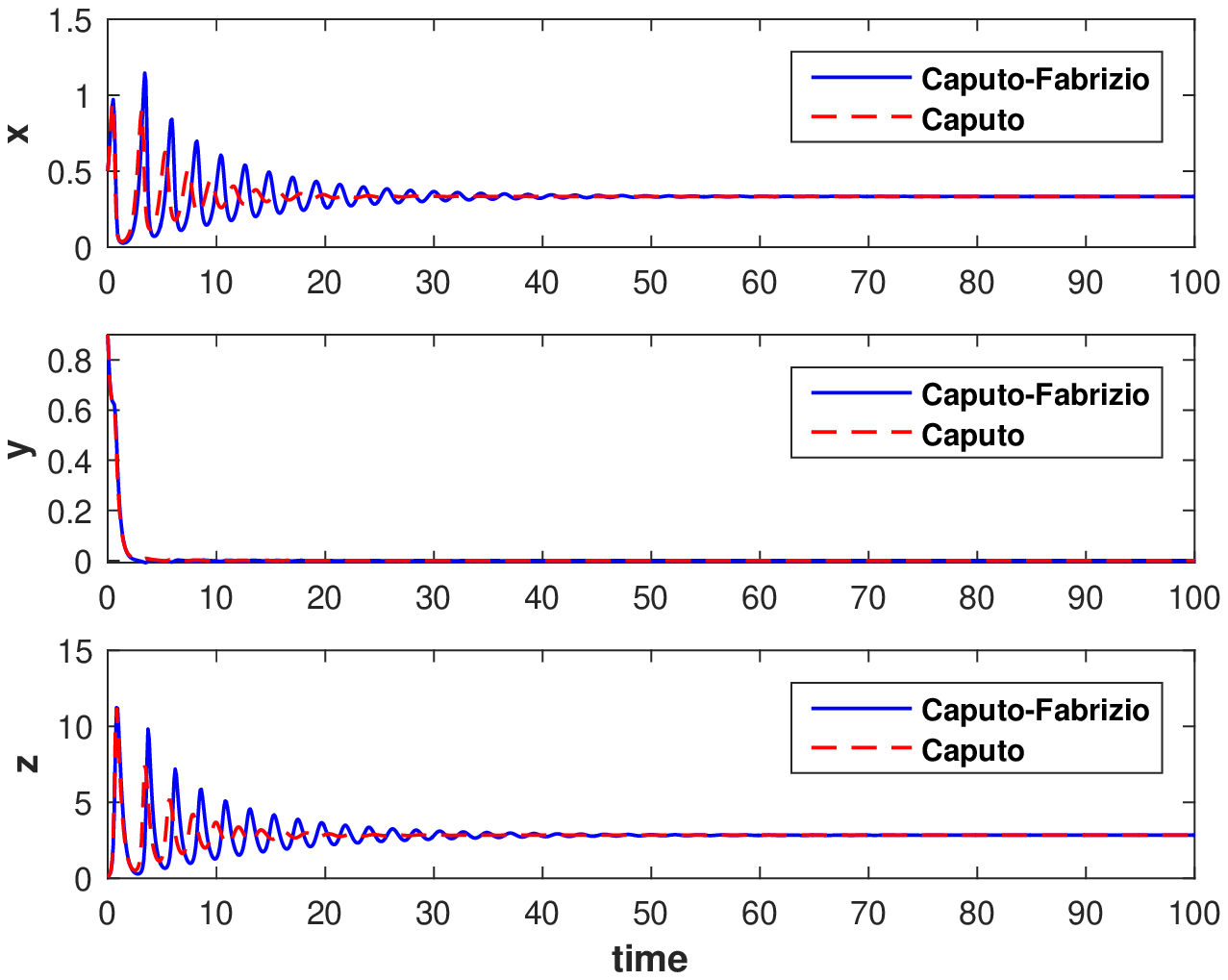} }}%
\subfloat{{\includegraphics[scale=.45]{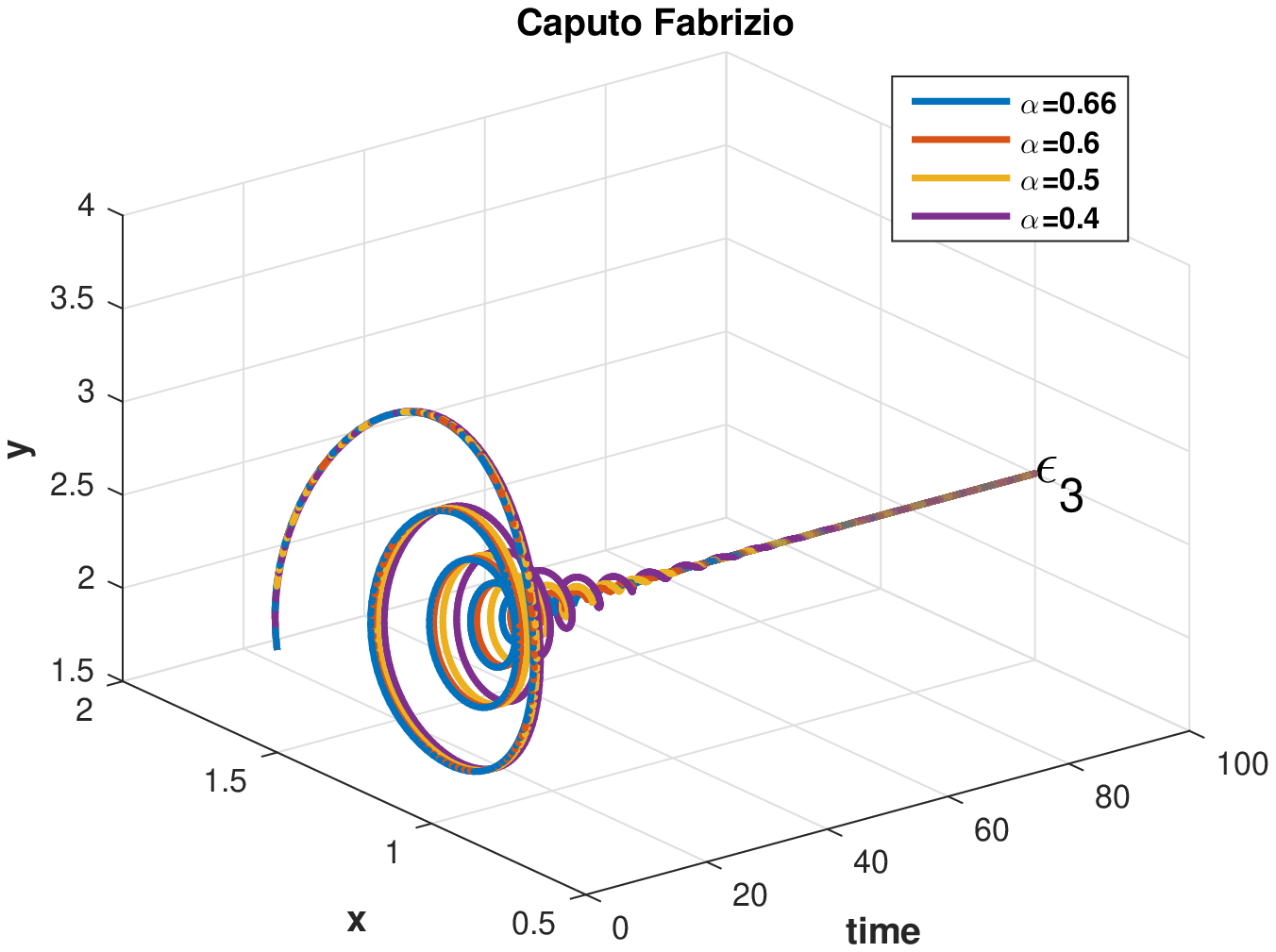} }}%
\caption{(left) Comparing the behavior of Caputo and CF operators for system \ref{e1} with the parameters of \hyperref[tab:table]{Example 1}, (right) converging to $\epsilon_3$ with CF operator for $\alpha\leq0.66$ and $(x_0,y_0,z_0)=(1.6,1.9,0)$.}%
\label{fig:example1}%
\end{figure}

\subsection{Example 2}
This example confirms the points mentioned in the previous one; by setting $\alpha=0.6$, the equilibrium $\epsilon_4$ is the only stable equilibrium point in the sense of Caputo derivative, while the equilibrium points $\epsilon_0$, $\epsilon_1$, and $\epsilon_3$ are stable concerning the CF operator. But, as shown in Fig. \ref{fig:compare}, it is interesting that we have here an eigenvalue, $\lambda_4$, in the class of $\lambda_B$ alongside the $\lambda_D$, where $\lambda_0$, $\lambda_1$, and $\lambda_3$ are. As a result, Fig. \ref{fig:example2} shows that the system can start from a point to converge asymptotically to the only equilibrium that is stable in the sense of Caputo, rather than the CF operator. Therefore, it could make a challenge for one who assumes a system having a more stability region may lead to more potential to achieve a steady-state.

\begin{figure}[!tbp]%
\centering
\subfloat{{\includegraphics[scale=.45]{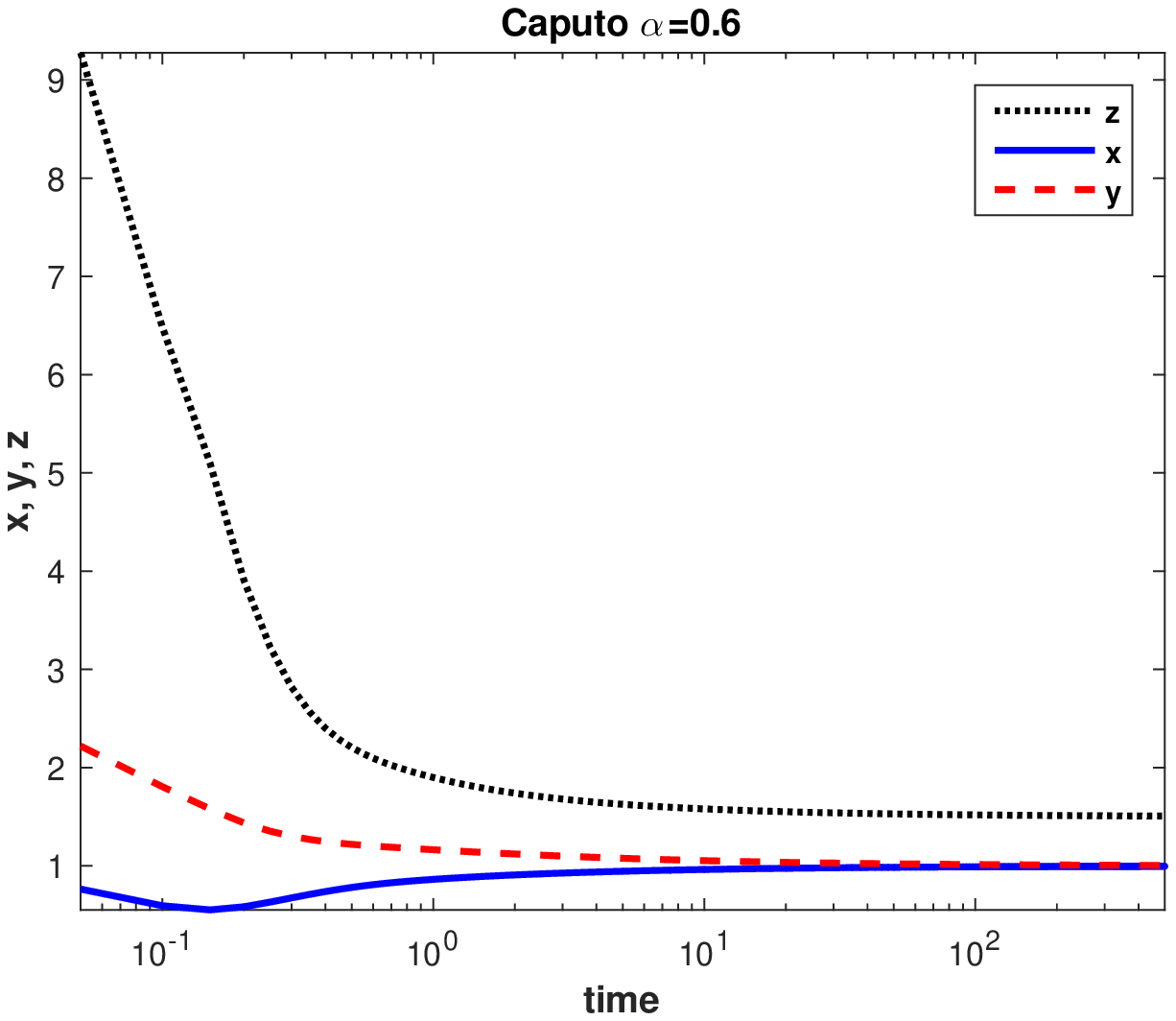} }}%
\subfloat{{\includegraphics[scale=.45]{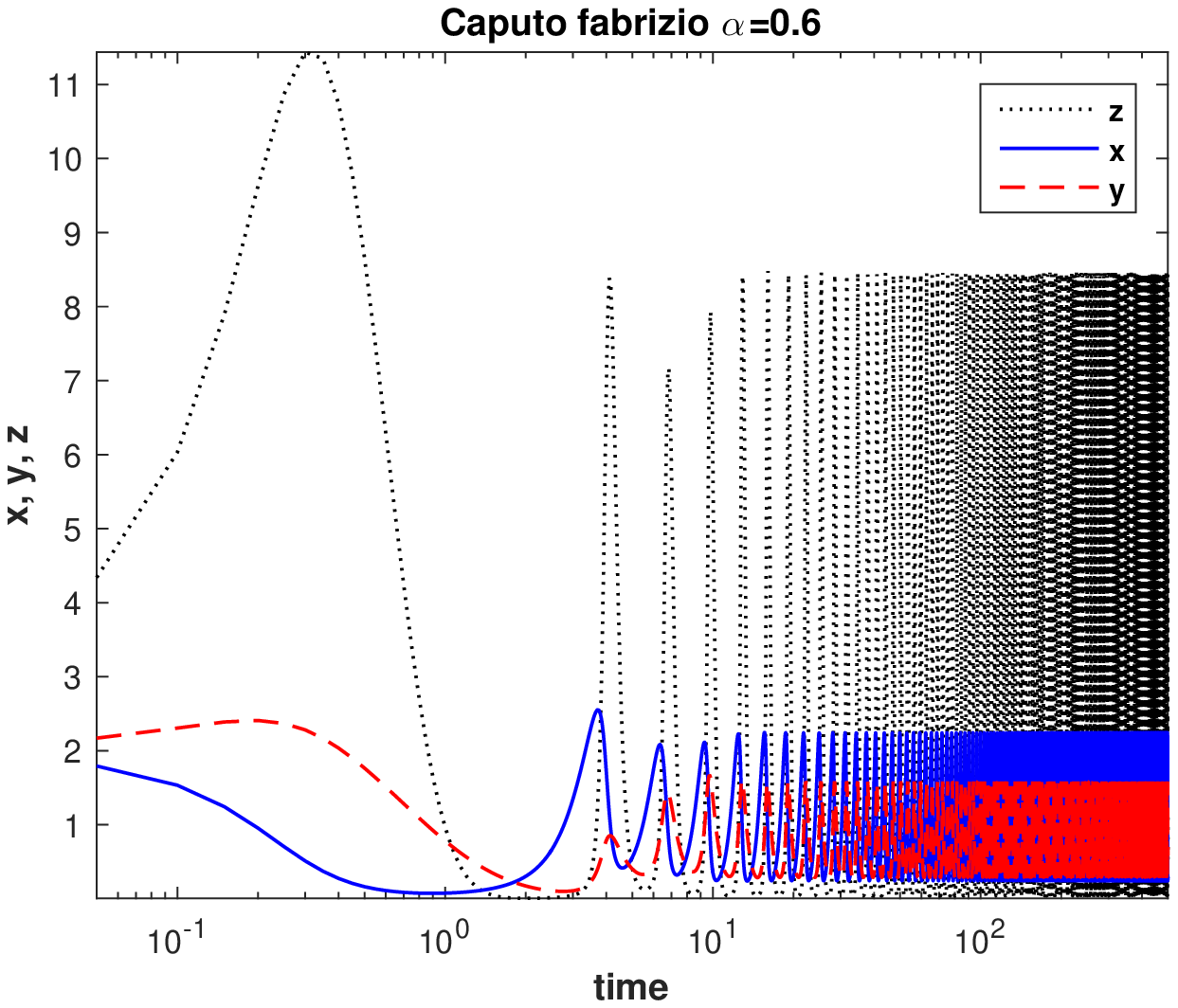} }}%
\caption{System \ref{e1} with the parameters of \hyperref[tab:table]{Example 2} and $(x_0,y_0,z_0)=(2,2,3)$ is asymptotically stable for Caputo (left) and unstable for CF (right).}
\label{fig:example2}%
\end{figure}

\begin{figure}[!tbp]%
\centering
\subfloat{{\includegraphics[scale=0.45]{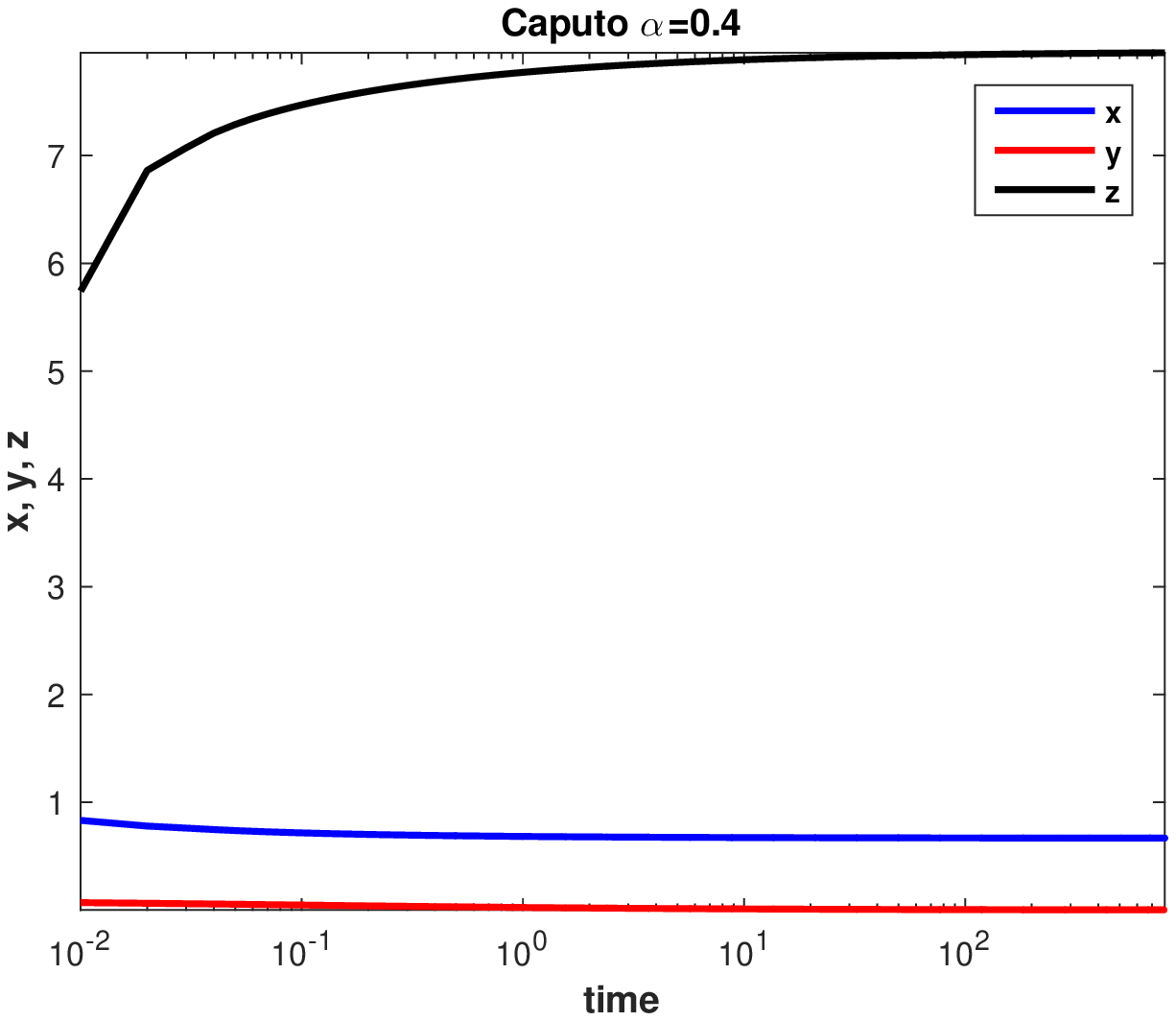} }}%
\subfloat{{\includegraphics[scale=0.45]{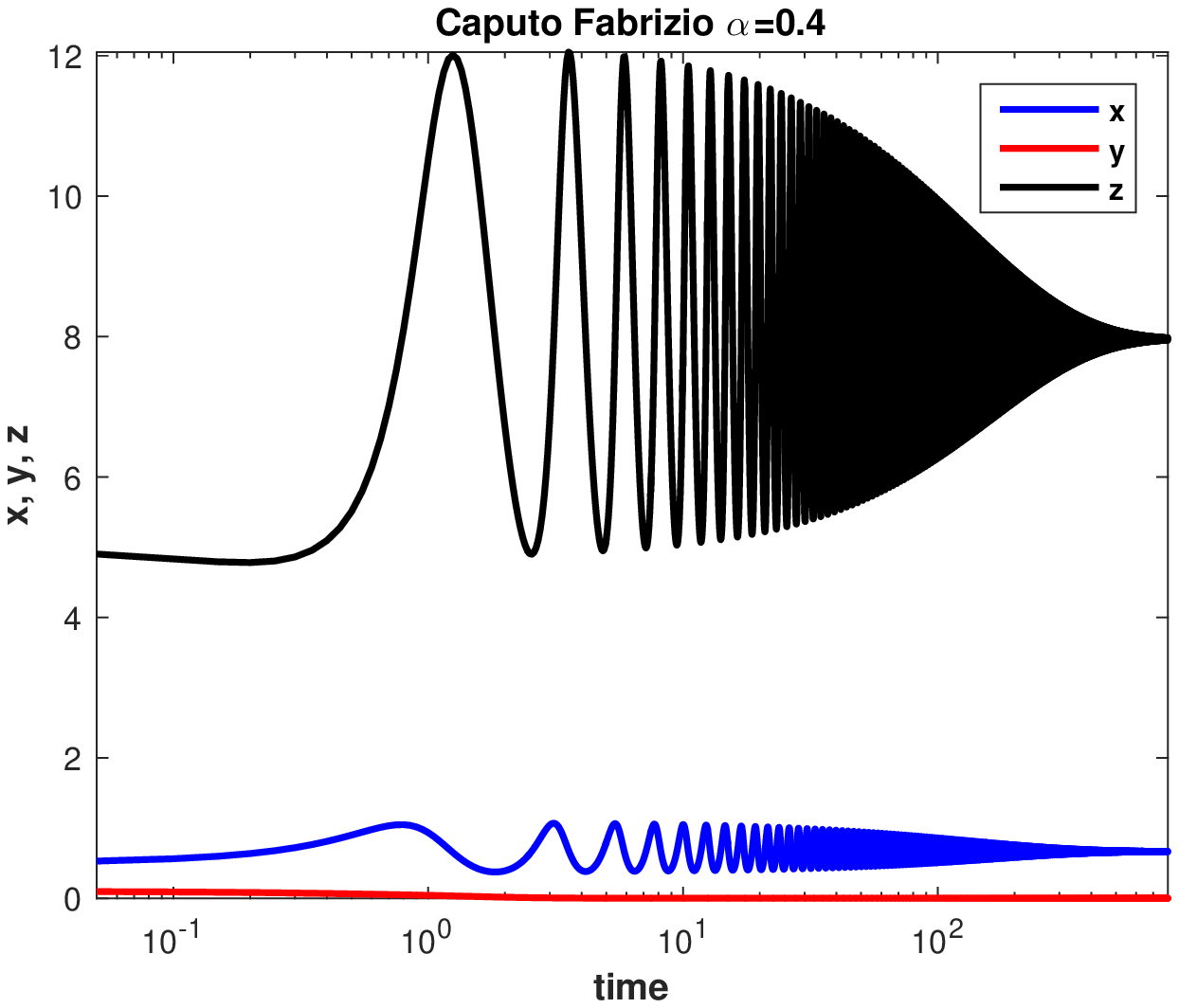} }}%
\caption{System \ref{e1} with the parameters of the \hyperref[tab:table]{Example 3} and $(x_0,y_0,z_0)=(0.5,0.1,5)$ is asymptotically stable for both Caputo (left) and CF (right) at $\epsilon_2$ with different oscillations.}%
\label{fig:example31}%
\end{figure}

\subsection{Example 3}

This example can complete the discussion and make clear the substantial role of the initial values on the behavior of the very Lotka-Volterra model. Considering the information in Table \ref{tab:table}, determining the location of four eigenvalues in the $\lambda_D$ class (Fig.~\ref{fig:compare}), we expect that the system is most stable for the CF operator and noticeably unstable for the Caputo derivative. Although Fig.~\ref{fig:example31} illustrates this expectation, it is not an absolute scenario when the process is supposed to start from a point leading to equilibrium $\epsilon_3$. It depends on the domain of attraction that the initial values stay and the corresponding eigenvalue, which is in the class of $\lambda_B$ (see Fig. \ref{fig:example32}). For more information on finding the domain of attractions to specific equilibrium points,
see \cite{attraction}.

\begin{figure}[!tbp]%
\centering
\subfloat{{\includegraphics[scale=.35]{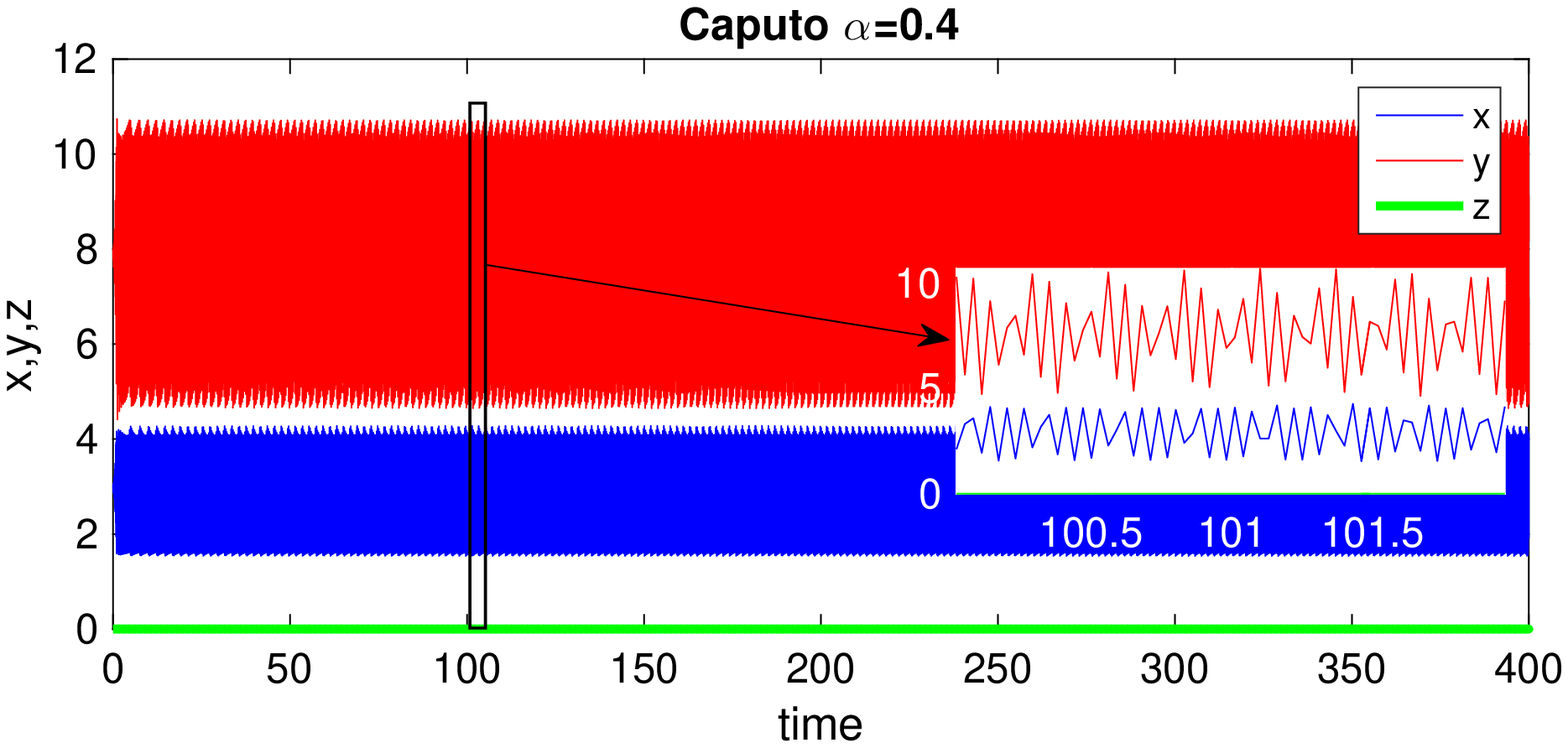} }}%
\subfloat{{\includegraphics[scale=.4]{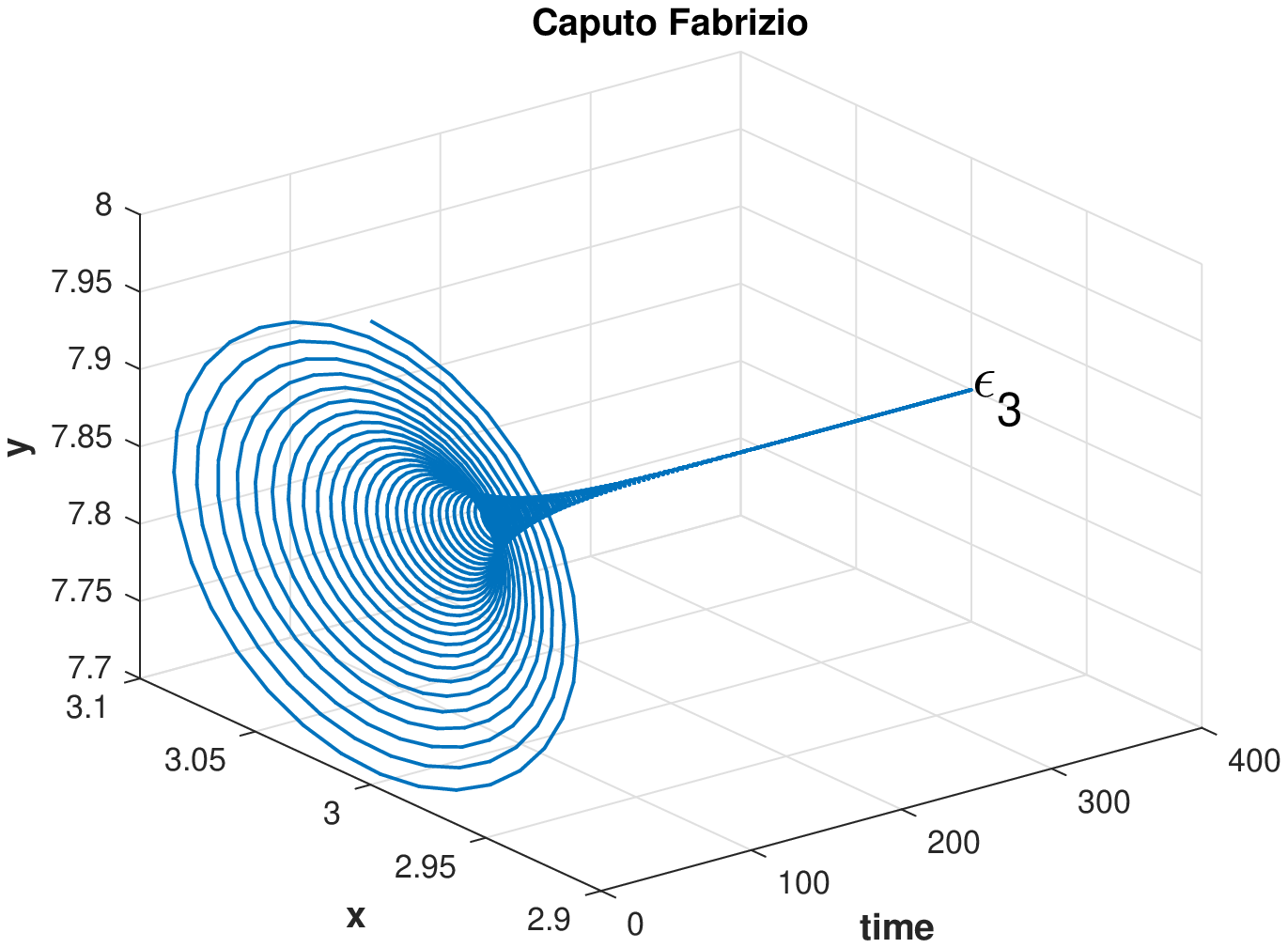} }}%
\caption{System \ref{e1} with the parameters of the \hyperref[tab:table]{Example 3} and $(x_0,y_0,z_0)=(3,8,0)$ is unstable for Caputo (left) and asymptotically stable for CF (right).}%
\label{fig:example32}%
\end{figure}

\section{Conclusion}

We have investigated the three-dimensional Lotka-Volterra system for the CF operator. Concerning the existence of a non-singular kernel in the definition of CF operator,  we investigated the stability of the system and suggested a new numerical method with improved stability properties based on Adams-Bashforth methods. This numerical scheme improves the efficiency of the simulations for both Caputo and CF operators. Numerical results demonstrate how the behavior of the Lotka-Volterra system can depend on the type of differential operator and the value of fractional order. Moreover, we have shown that the CF operator provides different properties compared to the classical Caputo derivative. Overall, this analysis can enhance our understanding of the exceptional dynamics in complex systems.

Analysing the behavior of Lotka-Volterra models under incommensurate fractional orders, where the different interacting partners may have different degrees of memory or lag effects, is a fascinating line for further research. In real-world complex systems, a time-variable dependency on the past states is likely and may lead to anomalous behaviors that pose challenges for modeling. Besides, fractional calculus provides tools to simulate systems with such incommensurate fractional-order derivatives.  Therefore, finding the stability region of the three-dimensional Lotka-Volterra model with incommensurate fractional orders in the sense of Caputo and CF operators as well as examining the domain of attractions are promising directions for future studies. 


\bibliography{ref}
\bibliographystyle{ieeetr}

\end{document}